\documentclass[12pt]{amsart}
\usepackage{amsmath,amssymb,amsfonts}
\usepackage[all]{xy}

\usepackage[T5,T1]{fontenc}
\usepackage{mathpazo}
\usepackage{hyperref}
\usepackage{a4wide}

\theoremstyle{plain}
\newtheorem{thm}{Theorem}[section]
\newtheorem{lem}[thm]{Lemma}
\newtheorem{cor}[thm]{Corollary}
\newtheorem{prop}[thm]{Proposition}

\theoremstyle{definition}

\newtheorem{rmk}[thm]{Remark}

\newcommand{\sA}{{\mathcal A}}
\newcommand{\sB}{{\mathcal B}}

\newcommand{\sL}{{\mathcal L}}

\newcommand{\F}{{\mathbb F}}

\newcommand{\U}{{\mathbb U}}

\begin{document}
\sloppy

\title{Zassenhaus filtrations as intersections} 
\dedicatory{Dedicated to Professor {\fontencoding{T5}\selectfont Ng\^o Vi\d\ecircumflex{}t Trung} on the occasion of his 70th birthday }

 \author{ J\'an Min\'a\v{c},  {\fontencoding{T5}\selectfont Nguy\~ \ecircumflex n Duy T\^an}  and {\fontencoding{T5}\selectfont Nguy\~ \ecircumflex n Th\d{i} Tr\`a}}

\address{Department of Mathematics, University of Western Ontario, London, Ontario, Canada N6A 5B7}
\email{minac@uwo.ca}
  \address{ Faculty of Mathematics and 	Informatics, Hanoi University of Science and Technology, 1 Dai Co Viet Road, Hanoi, Vietnam } 
\email{tan.nguyenduy@hust.edu.vn}

\address{Department of Mathematics, Hanoi Pedagogical No 2, Xuan Hoa, Vinh Phuc, Vietnam}
\email{nguyentra.bsu@gmail.com}

\thanks{ JM is partially supported  by the Natural Sciences and Engineering Research Council of Canada (NSERC) grant R0370A01. He gratefully acknowledges the Western University Faculty of Science Distinguished Professorship  and the support of the Western Academy for Advanced Research. NDT and NTT are partially supported by the National Foundation for Science and Technology Development (NAFOSTED) grant 101.04-2023.21}
 \begin{abstract}
Zassenhaus filtrations of profinite groups are an important tool to study profinite groups.
In this paper, we describe Zassenhaus filtrations of profinite groups as intersections of  kernels of certain representations. 
In this way we introduce a link between studying profinite groups with methods of Zassenhaus filtrations and representation theory.
\end{abstract}
\maketitle
\section{Introduction}
In our paper we study profinite groups, their Zassenhaus filtrations and certain representations of profinite groups. We refer the reader for the basic theory of profinite groups to  \cite{Ko,NSW,DDMS}. Profinite groups are essential in Galois theory. In fact each Galois group of a Galois extension is a profinite group. Conversely for each profinite group there exists a Galois field extension whose Galois group is isomorphic to the given profinite group. (See \cite{Wa}.)
Thus we see that we can view the theory of profinite groups as a part of Galois theory. In general profinite groups can be quite complicated and their study can be therefore challenging. In order to use methods of linear algebra, various descending series of profinite groups were introduced. We are looking for such filtrations where we can make related abelian subquotients. Then we can use basic linear algebra methods related to these abelian groups.

From now on $p$ stands for a prime number. All our subgroups of profinite groups will be always considered closed subgroups.
Let $G$ be a profinite group/ 
The \emph{$p$-Zassenhaus filtration} (or simply \emph{Zassenhaus filtration}, if the prime is clear from the context) of $G$ is 
defined inductively by
\begin{equation}\label{zassenhaus defi}
\begin{array}{lclr}
G_{(1)}&=&G& \\
 G_{(n)}&=&G_{(\lceil n/p\rceil)}^p\cdot\prod_{i+j=n}[G_{(i)},G_{(j)}], & n\geq2.
\end{array}
\end{equation}
Here $\lceil n/p\rceil$ is the least integer $h$ such that $hp\geq n$; and for two subgroups $H$ and $K$ of $G$, we denote $[H,K]$  the (closed) subgroup of $G$ generated by the commutators
\[ [x,y] = (y^{-1})^x\cdot y = x^{-1}y^{-1}xy , \qquad x\in H , y\in K .\]
Similarly, for $k\geq1$, $H^k$ is the (closed) subgroup of $G$ generated by the elements $g^k$, $g\in H$.

I. Efrat, in his beautiful pioneering paper \cite{Ef}, was able to  show the following result. 
\begin{thm}
\label{thm:Efrat}
Let $S$ be a free pro-$p$-group and $n\geq1$. Then $S_{(n)}$ is the intersection of all kernels of linear representations $\rho:G\to {\rm GL}_n(\F_p)$.
\end{thm}

In his proof of this result, I. Efrat uses Massey products and some duality principles between profinite groups and subgroups of cohomology groups of degree 2.
In \cite{MTE}, and independently in \cite{Ef2}, the authors provides another more direct proof of this result using Magnus homomorphisms.
Also in \cite{MT}, the authors formulate a conjecture that the above theorem is still true for the maximal pro-$p$ quotient of any field containing a primitive $p$th root of unity. This conjecture is called the "kernel $n$-unipotent conjecture". This conjecture was further investigated in \cite{MTE} and \cite{MT2}. In \cite{MTE} the conjecture was established for odd rigid fields and for Demushkin groups of rank 2. In \cite{MT2} the kernel unipotent conjecture was established for Demushkin finitely generated pro-$p$ groups in the case when $n=4$. We remark that in the appendix of \cite{MTE}, it is shown that for each $n\geq 3$, there are examples of pro-$p$ groups which do not satisfy the kernel $n$-unipotent property.

The aim of this paper is to describe $G_{(n)}$ as the kernel of some groups representations for a certain class of profinite groups $G$.   Our main result is the following theorem and one of the key ingredients of the proof is Theorem~\ref{thm:Efrat}.
\begin{thm}[=Theorem~\ref{thm:main}]
Let $n$ be an integer $\geq 1$. Let $S$ be a free pro-$p$-group and let $R$ be a normal subgroup of $S$ with $R\leq S_{(n)}$. Let $G=S/R$. Then \[ G_{(n+1)}=\bigcap \ker (\rho\colon G\to \U(\sA)),\] 
where the intersection is taken as $\sA$ runs over the collection of all rank $n$ multiplicative systems with $A_{1,n+1}=\F_p$ and $\rho$ runs over all homomorphism $G\to \U(\sA)$.
\end{thm}
See below, in Section~2, for the definition of a (rank~$n$) multiplicative system $\mathcal{A}$ and the group $\U(\mathcal{A})$. In Section~2 we also establish a convention that we suppress the subindex in $\U_{n,1}(\mathcal{A})$ and denote this group as simply $\U(\mathcal{A})$.

The proof of Theorem~\ref{thm:main} relies on Theorem~\ref{thm:kernel-pairing} and Corollary~\ref{cor:non-degenerate pairing}. In Theorem~\ref{thm:kernel-pairing}, we show that our desired conclusion in Theorem~\ref{thm:main} about a characterization of the $n$th term of the $p$-Zassenhaus filtration of $G$ in terms of the kernels of certain representations of $G$ is equivalent to a perfect pairing between subquotients of Zassenhaus filtration and certain subgroups of corresponding cohomology groups. A good part of the Section~3 is dedicated to a definition of the pairing and establishing results leading to Theorem~\ref{thm:kernel-pairing}.

\section{Multiplicative systems} Let $n$ be an integer $\geq 1$. In this section we use rank $n$  multiplicative systems which were introduced in \cite[page 183]{Dwy}.
Recall that a rank $n$ multiplicative system $\sA$ (of coefficient groups) is defined to be an array
\[
\sA=\{A_{ij}:1\leq i<j\leq n+1 \}
\]
of abelian groups, together with pairings
\[
\mu=\mu_{ijk}\colon A_{ij}\otimes A_{jk}\to A_{ik}
\]
which are associative in the sense that the following diagrams commute:
\[
\xymatrix{
A_{ij}\otimes A_{jk}\otimes A_{kl} \ar@{->}[r]^{1\otimes \mu} \ar@{->}[d]^{\mu\otimes 1} & A_{ij}\otimes A_{jl}\ar@{->}[d]^\mu\\
A_{ik}\otimes A_{kl} \ar@{->}[r]^\mu &A_{il}.
}
\]
We shall denote the pairing $\mu$ by juxtaposition of elements. 

\emph{ In this paper we  consider multiplicative systems of finite dimensional $\F_p$-vector spaces.}

For each integer $d\geq 1$, let $V_{n,d}(\sA)$  be the set of arrays
\[
\{r_{ij}\in A_{ij}\mid r_{ij}=0 \text{ if } j-i\leq d-1\}. 
\]
In particular, $V_{n,1}$ is the set of arrays $\{r_{ij}\in A_{ij}\mid 1\leq i<j\leq n+1\}$ (no restrictions on $r_{ij}$). 
One can check that $V_{n,d}(\sA)$ is an $\F_p$-algebra with addition and multiplication operations given by the following formulas
\[
\{r_{ij}\}+\{r^\prime_{ij}\}=\{t_{ij}\}, \text{ where } t_{ij}=r_{ij}+r^\prime_{ij},
\] 
and
\[
\{r_{ij}\}\cdot \{r_{ij}^\prime\}=\{u_{ij}\}, \text{ where } u_{ij}=\sum_{k=i+1}^{j-1} r_{ik} r^\prime_{kj} \text{ if } i<j-1 \text{ and } u_{ij}=0 \text{ if } i=j-1.
\]
Note that $V_{n,d}(\sA)V_{n,d^\prime}(\sA)\subseteq V_{n,d+d^\prime}(\sA)$ and that $V_{n,d}(\sA)=0$ for $d>n$. 

Let $\U_{n,d}(\sA)$ be the set of elements of the form $1+a$ with $a\in V_{n,d}(\sA)$. (Here $1+a$ is just a formal sum, and $1+a=1+b$ means exactly that $a=b$.) Then $\U_{n,d}(\sA)$ is a group with group operation given by
\[
(1+a)(1+b)=1+a+b+ab.
\]
Thus we see that the formal element $1$ introduced above indeed plays the role of the identity element in our group $\U_{n,d}(\sA)$.

We write simply $\U(\sA)$ for $\U_{n,1}(\sA)$. Set 
\[ Z(\sA)=\U_{n,n}(\sA)=1+V_{n,n}(\sA)=\{1+a\mid a\in V_{n,n}(\sA) \}.\] 
Then $Z(\sA)$ is a central subgroup of $\U(\sA)$. We define
$\overline{\U}(\sA)=\U(\sA)/Z(\sA)$. 
\begin{rmk}
If we choose $A_{ij} = \F_p$ for every $1 \leq i < j \leq n+1$, with the usual addition and multiplication, then $\U(\sA) \cong \U_{n+1}(\F_p)$. Here $\U_{n+1}(\mathbb{F}_p)$ is the group of upper triangular unipotent $(n+1) \times (n+1)$-matrices with entries in $\mathbb{F}_p$.
\end{rmk}

\begin{rmk}
\label{rmk:2.2}
Suppose that $n \geq 2$ and $\bar{\mathcal{A}} = \{ A_{ij} \}$ is a rank $n-1$ multiplicative system. Then there exists a rank $n$ multiplicative system $\mathcal{A}$ such that the group $\U(\bar{\mathcal{A}})$ can be embedded in the group $\U(\mathcal{A})$. This can be done as follows. Set $A_{i,n+1} = \mathbb{F}_p$ and set
\[
\mu_{ij,n+1} \colon A_{ij} \otimes A_{j,n+1} \to A_{i,n+1}
\]
to be the zero map, for each $1 \leq i < j \leq n$. Let $\mathcal{A} = \{ A_{ij} \mid 1 \leq i < j \leq n+1 \}$. Then $\mathcal{A}$ is a rank $n$ multiplicative system. Now for each
\[
\bar{a} = \{ r_{ij} \in A_{ij} \mid 1 \leq i < j \leq n \} \in V_{n-1,1}(\bar{\mathcal{A}}),
\]
set
\[
a = \{ r_{ij} \in A_{ij} \mid 1 \leq i < j \leq n+1 \} \in V_{n,1}(\mathcal{A}),
\]
where $r_{i,n+1} = 0$, $\forall\, 1 \leq i \leq n$. We define a map $\iota \colon \U(\bar{\mathcal{A}}) \to \U(\mathcal{A})$ by $\iota(1 + \bar{a}) = 1 + a$. Then $\iota$ is an injective group homomorphism.
\end{rmk}

Recall that for a profinite group $G$, the descending central series $(G_i)$ is defined inductively by 
\[
G_1=G, \quad G_{i+1}=[G_i,G].
\]
The following formula, which is due to M Lazard (see \cite[Theorem 11.2]{DDMS}), provides a relation between the $p$-Zassenhauss filtrations and the descending central series of a profinite group $G$,
\[
G_{(n)}=\prod\limits_{ip^k\geq n}G_i^{p^k}, \quad \text{ for all } n\geq 1. 
\]

\begin{lem}
\label{lem:2.3}
Let $d,d^\prime,k\geq 1$ be integers.
\begin{enumerate}
\item If $u \in \U_{n,d}(\sA)$ and $v\in \U_{n,d^\prime}(\sA)$, then $[u,v]\in \U_{n,d+d^\prime}$ and $u^{p^k}\in \U_{n,dp^k}(\sA)$.
\item One has $\U(\sA)_k \subseteq \U_{n,k}(\sA)$ and $\U_{n,d}(\sA)^{p^k} \subseteq \U_{n,dp^k}(\sA)$.
\item One has
$
\U(\sA)_{(k)}\subseteq \U_{n,k}(\sA). 
$
In particular, this implies that 
\[
\U(\sA)_{(n+1)}=1 \text{ and } \overline{\U}(\sA)_{(n)}=1. 
\]
\end{enumerate}
\end{lem}
\begin{proof} (1) Write $u=1+a$ and $v=1+b$, where $a\in V_{n,d}(\sA)$ and $b\in V_{n,d^\prime}(\sA)$. Then
\begin{align*}
[u^{-1},v^{-1}]&=(1+a)(1+b)(1+a)^{-1}(1+b)^{-1}\\
&=(1+a+b+ba + ab-ba)(1+b+a+ba)^{-1}\\
&=1+ (ab-ba)(1+\sum_{i=1}^\infty (-1)^i (a+b+ba)^i).
\end{align*}

Since $ab-ba\in V_{n,d+d^\prime}(\sA)$, we see that $[u,v]$ is in $V_{n,d+d^\prime}(\sA)$. 

We also have
$ u^{p^k}=(1+a)^{p^k}=1+a^{p^k}$, which  is an element of  $\U_{n,dp^k}(\sA)$.

(2) just follows from (1).

We are going to prove (3). For every $i,j\geq 1$ with  $p^j i\geq k$, one has
\[
\U(\sA)_i^{p^j}\subseteq \U_{n,i}(\sA)^{p^j}\subseteq \U_{n,ip^j}(\sA)\subseteq \U_{n,k}(\sA).
\]
Hence 
\[
\U(\sA)_{(k)}=\prod_{p^j i\geq k} \U(\sA)_i^{p^j}\subseteq \U_{n,k}(\sA).
\qedhere
\]
\end{proof}

\begin{lem}
\label{lem: trivial on Zassenhaus subgroups}
 Let $G$ be a pro-$p$-group. 
\begin{enumerate}
 \item Every continuous homomorphism $\rho\colon G\to \overline{\U}(\sA)$ is trivial on $G_{(n)}$.
\item Every continuous homomorphism $\rho\colon G\to \U(\sA)$ is trivial on $G_{(n+1)}$.
\end{enumerate}
\end{lem}
\begin{proof} The proof follows from Lemma~\ref{lem:2.3} (3).
\end{proof}

\begin{rmk}
\label{rmk:2.5}
From the above lemma, we see that for every pro-$p$ group $G$,
\[
G_{(n+1)} \subseteq \bigcap \ker(\rho \colon G \to \U(\mathcal{A})),
\]
where the intersection is taken as $\mathcal{A}$ runs over the collection of all rank $n$ multiplicative systems with $A_{1,n+1} = \mathbb{F}_p$ and $\rho$ runs over all homomorphisms $G \to \U(\mathcal{A})$. Furthermore, we shall show that the equality holds if either $G$ is free or $n = 1$.

Suppose $G$ is free. From Theorem~\ref{thm:Efrat} we see that
\[
\bigcap \ker(\rho \colon G \to \U(\mathcal{A})) \subseteq \bigcap \ker(\rho \colon G \to \U_{n+1}(\mathbb{F}_p)) = G_{(n+1)}.
\]
Hence the equality holds.

Now we suppose that $n = 1$. Let $\sigma \in G \setminus G_{(2)}$. Then $\bar{\sigma} \in G/G_{(2)}$ is nonzero. Let $\mathcal{B}$ be a basis for $\mathbb{F}_p$-vector space $G/G_{(2)}$ that contains $\bar{\sigma}$. Let $f \colon G/G_{(2)} \to \mathbb{F}_p$ be the linear mapping such that $f(\bar{\sigma}) = 1$ and $f(x) = 0$ for every $x \in \mathcal{B} \setminus \{\bar{\sigma}\}$. We define a homomorphism $\bar{\rho} \colon G/G_{(2)} \to \U_2(\mathbb{F}_p)$ by
\[
\bar{\rho}(\bar{\tau}) = \begin{bmatrix} 1 & f(\bar{\tau}) \\ 0 & 1 \end{bmatrix}, \quad \text{ for all }  \bar{\tau} \in G/G_{(2)}.
\]
Let $\rho \colon G \to \U_2(\mathbb{F}_p)$ be the composition of $\bar{\rho}$ and the natural homomorphism $G \to G/G_{(2)}$. Then
\[
\rho(\sigma) = \begin{bmatrix} 1 & 1 \\ 0 & 1 \end{bmatrix}.
\]
Hence $\sigma \notin \ker(\rho)$. The above argument implies that
\[
\bigcap \ker(\rho \colon G \to \U(\mathcal{A})) \subseteq \bigcap \ker(\rho \colon G \to \U_2(\mathbb{F}_p)) \subseteq G_{(2)},
\]
and the equality holds.
\end{rmk}

For a representation $\rho\colon G\to \U(\sA)$ and $1\leq i\leq j\leq n+1$,
let $\rho_{ij}\colon G\to A_{i,j}$
be the composition of $\rho$ with the projection from $\U(\sA)$ to its $(i,j)$-coordinate.
We use a similar notation for representations $\bar\rho\colon G\to\overline{\U}(\sA)$.
Note that $\rho_{i,i+1}$ (resp., $\bar\rho_{i,i+1}$) is a group homomorphism.

Suppose that $n\geq 2$. Let $\sA=\{A_{ij}\}$ be a rank $n$ multiplicative system (of finite dimensional $\F_p$-vector spaces). Let $G$ be a pro-$p$ group which acts trivially on each $A_{ij}$. We choose classes $\alpha_i\in H^1(G,A_{i,i+1})$, and suppose that there is an array $M$ of cochains
\[
M=\{a_{ij}\mid 1\leq i<j\leq n+1, (i,j)\not=(1,n+1), a_{ij}\in C^1(G,A_{ij})\} 
\]
such that the following conditions are fulfilled: 
\begin{enumerate}
\item $a_{i,i+1}\in Z^1(G,A_{i,i+1})$ represents $\alpha_i$,
\item  $\partial a_{ij}=\sum_{k=i+1}^{j-1} a_{ik}\cup a_{kj}$ for $i+1<j$.
\end{enumerate}
Such an $M$ is called a {\it defining system} for the (generalized) $n$th Massey product $\langle \alpha_1,\ldots,\alpha_n\rangle^{\sA}$. For a defining system $M=\{a_{ij}\}$, one can check that $\sum_{k=2}^n a_{1,k}\cup a_{k,n+1}$ is a 2-cocycle in $C^2(G,A_{1,n+1})$. 
 The cohomology class in $H^2(G,A_{1,n+1})$ of this 2-cocycle is called the {\it value} of the Massey product relative to $M$, and is denoted by $\langle \alpha_1,\ldots,\alpha_n\rangle^{\sA}_M$.

\begin{thm}[{\cite[Theorem 2.6]{Dwy}}]
\label{thm:Dwyer}
Suppose $\alpha_i \in H^1(G,A_{i,i+1})$, $1\leq i\leq n$. There is a one-one correspondence $M\leftrightarrow \bar\rho_M$ between defining systems $M$ for $\langle \alpha_1,\ldots,\alpha_n\rangle^\sA$ and group homomorphism $\bar\rho_M:G\to \overline{\U}(\sA)$ with $(\bar\rho_M)_{i,i+1}= -\alpha_i$, for $1\leq i\leq n$.

Moreover $\langle \alpha_1,\ldots,\alpha_n\rangle_M^\sA=0$ in $H^2(G,A_{1,n+1})$ if and only if the dotted arrow exists in the following  commutative diagram
\[
\xymatrix{
& & &G \ar@{->}[d]^{\bar\rho_M} \ar@{-->}[ld]\\
0\ar[r]& A_{1,n+1}\ar[r] &\U(\sA)\ar[r] &\overline{\U}(\sA)\ar[r] &1.
}
\]
\end{thm}
Explicitly, the one-one correspondence in Theorem~\ref{thm:Dwyer} is given as follows: For a defining system $M=(a_{ij})$ for $\langle \alpha_1,\ldots,\alpha_n\rangle$, $\rho_M:G\to \overline{\U}_{n+1}(A)$ is given by letting $(\bar\rho_M)_{ij}=-a_{ij}$.

As in \cite[Remark, p. 182]{Dwy}, one has the following result.
\begin{prop}
\label{prop:Dwyer}
Let $\bar\rho_M\colon G\to \overline{\U}(\sA)$ correspond to a defining system $M=(c_{ij})$ for $\langle \alpha_1,\ldots,\alpha_n\rangle^\sA$ as in Theorem~\ref{thm:Dwyer}. Then the central extension associated with $\langle \alpha_1,\ldots,\alpha_n\rangle^\sA_M$ is the pull back
\[
0\to A_{1,n+1} \to \U(\sA)\times_{\overline{\U}(\sA)} G \to G \to 1
\]
via $\bar\rho_M\colon G\to \overline{\U}(\sA)$ of the extension
\[0\to A_{1,n+1} \to \U(\sA)\to \bar \U(\sA) \to 1.\]
\end{prop}

Let $A$ be a finite-dimensional $\F_p$-vector space. 
A class $\alpha\in H^2(G,A)$ is said to be {\it decomposable} if there is some rank $n$ multiplicative system $\sA$ of $\F_p$-vector spaces  with $A_{1,n+1}=A$,  some choice of classes $\alpha_i\in H^1(G,A_{i,i+1})$ and some defining system $M$ for $\langle \alpha_1,\ldots,\alpha_n\rangle^{\sA}$ such that $\alpha= \langle \alpha_1,\ldots,\alpha_n\rangle^{\sA}_M$.
The set $\Phi^n(G,A)$ is defined as the subset of $H^2(G,A)$ consisting of all such decomposable classes.

\begin{lem}
$\Phi^n(G,A)$ is a subgroup of $H^2(G,A)$.
\end{lem}
\begin{proof}  Let $\alpha$ and $\alpha^\prime$ be two elements in $\Phi^n(G,A)$. Then $\alpha= \langle \alpha_1,\ldots,\alpha_n\rangle^{\sA}_M$ for some rank $n$ multiplicative system $\sA=\{A_{ij}\}$ with $A_{1,n+1}=A$ with pairings $\mu_{ijk}\colon A_{ij}\otimes A_{jk}\to A_{ik}$, some choice of classes $\alpha_i\in H^1(G,A_{i,i+1})$ and some defining system $M=\{a_{ij}\}$ for $\langle \alpha_1,\ldots,\alpha_n\rangle^{\sA}$.
 
Similarly, $\alpha^\prime= \langle \alpha^\prime_1,\ldots,\alpha_n^\prime\rangle^{\sA'}_{M^\prime}$ for some rank $n$ multiplicative system $\sA^\prime=\{A^\prime_{ij}\}$ with $A^\prime_{1,n+1}=A$ with pairings $\mu'_{ijk}\colon A'_{ij}\otimes A'_{jk}\to A'_{ik}$, some choice of classes $\alpha^\prime_i\in H^1(G,A^\prime_{i,i+1})$ and some defining system $M^\prime=\{a^\prime_{ij}\}$ for $\langle \alpha^\prime_1,\ldots,\alpha_n^\prime\rangle^{\sA^\prime}$. 

We set $B_{1,n+1}=A$, and for  other pairs $(i,j)$ we set $B_{ij}=A_{ij}\oplus A_{ij}^\prime$. We define $\nu_{ijk}\colon B_{ij}\otimes B_{jk}\to B_{ik}$ by
\[
\nu_{ijk}((a+b)\otimes(c+d))= \mu_{ijk}(a\otimes c)+\mu_{ijk}^\prime(b\otimes d).
\] 
We set $\beta_{i}=\alpha_i+\alpha_i^\prime$ and $b_{ij}=a_{ij}+a_{ij}^\prime\in C^1(G,B_{ij})$. Then  $\sB=\{B_{i,j}\}$ is a multiplicative system of rank $n$, and $N=\{b_{ij}\}$ is a defining system for $\langle \beta_1,\ldots,\beta_n\rangle^{\sB}$ and $\alpha+\alpha^\prime=\langle \beta_1,\ldots,\beta_n\rangle^{\sB}_N$.

Since $A$ is an $\F_p$-vector space, $H^2(G,A)$ is killed by $p$. Hence $-\alpha=(p-1)\alpha$ is  in $\Phi^n(G,A)$.
So $\Phi^n(G,A)$  is also closed under taking (additive) inverse. Therefore, $\Phi^n(G,A)$ is a subgroup of $H^2(G,A)$.
\end{proof}

\section{Pairings} 
\emph{In this section we suppose that $n$ is an integer $\geq 2$, except  in Theorem~\ref{thm:main}.}

Let $\sA$ be a multiplicative system of rank $n$ with $A_{1,n+1}=\F_p$. Denote $H^i(G)=H^i(G,\F_p)$ and $\Phi^n(G)=\Phi^n(G,\F_p)$. We always assume that the action of $G$ on $\F_p$ is
trivial. Observe that in our main applications $G$ is a pro-$p$ group and any pro-p group
can act on $\F_p$ only trivially.

Consider the exact sequence of profinite groups 
\[
1\to N\to G\to G/N\to 1.
\]
One has the 5-term exact sequence for the lower cohomology groups with coefficients in $\F_p$
\[
0\to H^1(G/N) \stackrel{{\rm inf}}{\to} H^1(G)\stackrel{{\rm res}}{\to} H^1(N)^{G/N}\stackrel{{\rm trg}}{\to} H^2(G/N)\stackrel{{\rm inf}}{\to} H^2(G).
\]

Let $\alpha$ be any element in $\Phi^n(G/N)$. Then $\alpha = \langle \alpha_1, \ldots, \alpha_n \rangle^{\mathcal{A}}_M$ for some multiplicative system $\mathcal{A} = \{A_{ij}\}$, some classes $\alpha_i \in Z^1(G/N, A_{i,i+1})$ and some defining system $M = \{a_{ij}\}$. Let $b_{ij}$ be the composition of the natural map $G \to G/N$ and $a_{ij}$. Let $\beta_i$ be the composition of the natural map $G \to G/N$ and $\alpha_i$. Then $M' = \{b_{ij}\}$ is a defining system for the Massey product $\langle \beta_1, \ldots, \beta_n \rangle^{\mathcal{A}}$, and $\beta := \langle \beta_1, \ldots, \beta_n \rangle^{\mathcal{A}}_{M'} = \inf(\alpha)$. Hence the inflation induces a homomorphism $\Phi^n(G/N) \to \Phi^n(G)$. Clearly, 
\[\ker(\Phi^n(G/N) \to \Phi^n(G)) = \ker(\Phi^n(G/N) \to H^2(G)).\]

Now assume that $N\leq G^p[G,G]=G_{(2)}$. Then the inflation map $H^1(G/N) \stackrel{{\rm inf}}{\to} H^1(G)$ is surjective. This is because each homomorphism $\varphi \colon G \to \F_p$ vanishes on $G^p[G, G]$ and hence also on $N$.  Hence 
\[
{\rm trg}\colon H^1(N)^{G/N}\to \ker(H^2(G/N)\stackrel{{\rm inf}}{\to} H^2(G))
\]
is an isomorphism. 


We assume further that $N\leq G_{(n)}$. Observe that $G/N$ acts trivially on $H^1(N/N \cap G_{
(n+1)})$ because for each $g$ in $G$ and $h$ in $N$, $g^{-1}h^{-1}gh$ belongs to $G_{(n+1)}$.
By applying the above discussion to the short exact sequence
\[
1\to N/N\cap G_{(n+1)}\to G/N\cap G_{(n+1)}\to G/N\to 1,
\] 
we obtain another isomorphism, still denoted by ${\rm trg}$:
\[
{\rm trg}\colon H^1(N/N \cap G_{
(n+1)})\to \ker(H^2(G/N)\stackrel{{\rm inf}}{\to} H^2(G/N \cap G_{
(n+1)})).
\]
We then can define a pairing
\[ \langle \cdot,\cdot\rangle \colon N/N\cap G_{(n+1)} \times \ker (\Phi^n(G/N)\to \Phi^n(G/N\cap G_{(n+1)}))\to \F_p
\]
by $\langle \bar{\sigma},\alpha\rangle=({\rm trg}^{-1}(\alpha))(\bar{\sigma})$, for $\bar{\sigma}\in N/N\cap G_{(n+1)}$, $\alpha\in  \ker (\Phi^n(G/N)\to \Phi^n(G/N\cap G_{(n+1)})) $. 
From the definition, we see that $\langle \bar{\sigma},\alpha\rangle$ is additive with respect to both $\alpha$ and $\bar{\sigma}$.
 Also $N/(N \cap G_{(n+1)})$ is an elementary abelian $p$-group, that means $N/(N \cap G_{(n+1)})$ is an $\mathbb{F}_p$-vector space.

\begin{lem}
\label{lem:3.1}
One has
\[
\begin{aligned}
\ker(\Phi^n(G/N) \to \Phi^n(G/N \cap G_{(n+1)})) &= \ker(\Phi^n(G/N) \to H^2(G/N \cap G_{(n+1)}))\\& = \ker(\Phi^n(G/N) \to H^2(G)).
\end{aligned}
\]
\end{lem}

\begin{proof}
We only need to show that
\[
\ker(\Phi^n(G/N) \to H^2(G)) \subseteq \ker(\Phi^n(G/N) \to H^2(G/N \cap G_{(n+1)})).
\]
Let $\alpha$ be an arbitrary element in $\ker(\Phi^n(G/N) \to H^2(G))$. Then $\alpha = \langle \alpha_1, \ldots, \alpha_n \rangle^{\mathcal{A}}_{M'}$ for some multiplicative system $\mathcal{A}$, some classes $\alpha_i$, and some defining system $M$. Let $\bar{\rho} \colon G/N \to \overline{\U}(\mathcal{A})$ be the homomorphism corresponding to the defining system $M$. Since $\inf(\alpha) = 0 \in H^2(G)$, by Hoechsmann’s lemma (\cite[Proposition 3.5.9]{NSW}), there exists a lift $\rho \colon G \to \U(\mathcal{A})$ of $\bar{\rho}$, i.e., the following diagram
\[
\xymatrix{
G \ar@{->}[r] \ar@{->}[d]^{\rho} & G/N \ar@{->}[r] \ar@{->}[d]^{\bar{\rho}} & 1 \\
\U(\mathcal{A}) \ar@{->}[r] & \overline{\U}(\mathcal{A}) \ar@{->}[r] & 1
}
\]
is commutative.

By Lemma~\ref{lem: trivial on Zassenhaus subgroups}, $\rho$ factors through a homomorphism $\rho' \colon G/N \cap G_{(n+1)} \to \U(\mathcal{A})$. Therefore we obtain the following commutative diagram
\[
\xymatrix{
G \ar@{->}[r] \ar@{->}[d]^{\rho} & G/N \cap G_{(n+1)} \ar@{->}[r] \ar@{->}[d]^{\rho'} & G/N \ar@{->}[r] \ar@{->}[d]^{\bar{\rho}} & 1 \\
\U(\mathcal{A}) \ar@{=}[r] & \U(\mathcal{A}) \ar@{->}[r] & \overline{\U}(\mathcal{A}) \ar@{->}[r] & 1.
}
\]

Then $\rho'$ is a lift of $\bar{\rho}$. This implies that $\inf(\alpha) = 0$ in $H^2(G/N \cap G_{(n+1)})$, and we are done.
\end{proof}

The above pairing can also be described as follows. Let $\alpha=\langle \alpha_1,\ldots,\alpha_n\rangle^{\sA}_M$ and
 let $\bar\rho\colon G/N\to \overline{\U}(\sA)$ be the homomorphism corresponding to the defining system $M$. Since ${\rm inf}(\alpha)=0\in \Phi^n(G/N\cap G_{(n+1)})\subseteq H^2(G/N\cap G_{(n+1)})$, there exists a lift $\rho\colon G/N\cap G_{(n+1)}\to \U(\sA)$ of $\bar\rho$, i.e.,  the following diagram
\[
\xymatrix{
N/N\cap G_{(n+1)}\ar@{->}[r] & G/N\cap G_{(n+1)} \ar@{->}[d]^{\rho} \ar@{->}[r] &G/N \ar@{->}[d]^{\bar\rho}\ar@{->}[r] &1\\
\F_p\ar@{->}[r] &\U(\sA)\ar[r] &\overline{\U}(\sA)\ar[r] &1
}
\]
is commutative.  Then the description of $\langle \bar{\sigma},\alpha\rangle$ is given by the next lemma.
\begin{lem}
\label{lem: pairing}
 One has $\langle \bar{\sigma},\alpha\rangle= -\rho_{1,n+1}(\bar{\sigma})$.
\end{lem}
\begin{proof} For any $\bar{\tau}\in  N/N\cap G_{(n+1)}$,  $\rho(\bar{\tau})$ is in $\ker(\U(\sA)\to\overline{\U}(\sA))$. Hence, we can write $\rho(\bar{\tau})=1+a(\bar{\tau})$, where $a(\bar{\tau})\in V_{n,n}(\sA)$.
For any $\bar{\sigma},\bar{\tau} \in N/N\cap G_{(n+1)}$, one has
\[
\begin{aligned}
1+a(\bar{\sigma}\bar{\tau}) &=\rho(\bar{\sigma}\bar{\tau})=\rho(\bar{\sigma})\rho(\bar{\tau})=(1+a(\bar{\sigma}))(1+a(\bar{\tau}))=1+a(\bar{\sigma})+a(\bar{\tau})+a(\bar{\sigma})a(\bar{\tau})\\
&=1+a(\bar{\sigma})+a(\bar{\tau}).
\end{aligned}
\] 
This implies that $a(\bar{\sigma}\bar{\tau})=a(\bar{\sigma})+a(\bar{\tau})$, for any $\bar{\sigma}$ and $\bar{\tau}$ in $N/N\cap G_{(n+1)}$. Thus, we can define $\Lambda(\rho)\in H^1(N/N\cap G_{(n+1)})={\rm Hom}(N/N\cap G_{(n+1)},\F_p)$ by 
\[
\Lambda(\rho)(\bar{\tau})=-\rho_{1,n+1}(\bar{\tau}), \text{ for }\bar{\tau} \in N/N\cap G_{(n+1)}.
\]
By \cite[Lemma 3.7]{MT}, one has ${\rm trg}(\Lambda(\rho))=\langle -\bar{\rho}_{1,2},\ldots,-\bar{\rho}_{n,n+1}\rangle^{\sA}_M=\alpha$. Hence 
\[
\langle \bar{\sigma}, \alpha \rangle=({\rm trg}^{-1}(\alpha))(\bar{\sigma}) = \Lambda(\rho)(\bar{\sigma})=-\rho_{1,n+1}(\bar{\sigma}).
\qedhere
\]
\end{proof}

\begin{lem} 
\label{lem: right non-degenerate}
 The pairing
\[N/N\cap G_{(n+1)} \times \ker (\Phi^n(G/N)\to \Phi^n(G/N\cap G_{(n+1)}))\to \F_p\] 
is right non-degenerate.
\end{lem}
\begin{proof} Given $\alpha \in \ker (\Phi^n(G/N)\to \Phi^n(G/N\cap G_{(n+1)}))$ such that $\langle\bar{\sigma},\alpha\rangle=0$ for every $\bar{\sigma}\in N/N\cap G_{(n+1)}$, we need to show that $\alpha=0$. 
 One has $\alpha=\langle \alpha_1,\ldots,\alpha_n\rangle^{\sA}_M$, for some multiplicative system $\sA$, some classes $\alpha_i$ and some defining system $M$. 
Let $\bar\rho\colon G/N\to \overline\U(\sA)$ be the homomorphism corresponding to the defining system $M$. We have already seen that there exists a lift $\rho\colon G/N\cap G_{(n+1)}\to \U(\sA)$ of $\bar\rho$. By Lemma~\ref{lem: pairing}, for every $\bar{\sigma}\in N/N\cap G_{(n+1)}$ one has
\[
\rho_{1,n+1}(\bar{\sigma})=- \langle \bar{\sigma}, \alpha\rangle =0. 
\]
This implies that $\rho(\bar{\sigma})=1$ for every $\bar{\sigma} \in N/N\cap G_{(n+1)}$. (Note that we assume that $N\subseteq G_{(n)}$, hence $\bar\rho$ is trivial on $N$ by Lemma~\ref{lem: trivial on Zassenhaus subgroups}.)
Thus $\rho$ factors through a homomorphism $\rho^\prime\colon G/N\to \U(\sA)$:
\[
\xymatrix{
N/N\cap G_{(n+1)}\ar@{->}[r]    & G/N\cap G_{(n+1)} \ar@{->}[d]^{\rho} \ar@{->}[r] &G/N \ar@{->}[d]^{\bar\rho}\ar@{->}[r] \ar@{->}[ld]^{\rho^\prime} &1\\
\F_p\ar@{->}[r] &\U(\sA)\ar[r] &\overline{\U}(\sA)\ar[r] &1
}
\]
This implies that $\alpha=\langle \alpha_1,\ldots,\alpha_n\rangle^{\sA}_M=0$.
\end{proof}

\begin{prop} 
\label{prop: non-degenerate and interesection}
 Let $G$ be a profinite group. Consider the following statements.
\begin{enumerate}
\item For every normal subgroup $N$  of $G$ such that $N\leq G_{(n)}$, the pairing
\[N/N\cap G_{(n+1)} \times \ker (\Phi^n(G/N)\to \Phi^n(G/N\cap G_{(n+1)}))\to \F_p\] 
is non-degenerate.
\item  $G_{(n+1)}=\cap \ker (\rho:G\to \U(\sA))$, where the intersection is taken as $\sA$ runs over the collection of all rank $n$ multiplicative systems with $A_{1,n+1}=\F_p$ and $\rho$ runs over all homomorphism $G\to \U(\sA)$.
\end{enumerate}
Then $(2)$ implies $(1)$.
\end{prop}

\begin{proof} Suppose that $(2)$ holds. 
 By Lemma~\ref{lem: right non-degenerate} it is enough to show that the pairing is left non-degenerate. 
For each $\sigma \in N$, we denote $\bar{\sigma}$ its image in the quotient group $N/N \cap G_{(n+1)}$.

Suppose that $\sigma\in N$ such that $\langle \bar\sigma,\alpha\rangle=0$, for every 
$\alpha\in \ker(\Phi^n(G/N)\to \Phi^n(G/N\cap G_{(n+1)}))$.  Let $\sA$ be any rank $n$ multiplicative system  with $A_{1,n+1}=\F_p$ and $\rho_0\colon G\to \U(\sA)$ any homomorphism. 
By Lemma~\ref{lem: trivial on Zassenhaus subgroups}, $\rho_0$ factors through $\rho\colon G/N\cap G_{(n+1)}\to \U(\sA)$, and the composition $G/N\cap G_{(n+1)}\to\U(\sA) \to \overline{\U}(\sA)$ factors through $\bar\rho\colon G/N \to \overline\U(\sA)$, so that the following diagram is commutative
\[
\xymatrix{
G/N\cap G_{(n+1)} \ar@{->}[r] \ar@{->}[d]_{\rho} & G/N \ar@{->}[r] \ar@{->}[d]_{\bar\rho} & 1\\
\U(\sA) \ar@{->}[r] & \overline{\U}(\sA)\to 1.
}
\]
Let $\alpha\in \Phi^n(G/N)$ be the element corresponding to $\bar\rho$. Then ${\rm inf}(\alpha)=0 $ in $H^2(G/N\cap G_{(n+1)})$ since $\rho$ is a lift of $\bar\rho$.
Thus by Lemma~\ref{lem:3.1}, $\alpha$ is in $\ker(\Phi^n(G/N)\to \Phi^n(G/N\cap G_{(n+1)}))$.

On the other hand, by Lemma~\ref{lem: pairing} 
\[
\langle \bar \sigma, \alpha\rangle = -\rho(\bar\sigma)_{1,n+1}.
\]
Hence $\rho(\bar{\sigma})_{1,n+1}=0$.

By Lemma~\ref{lem: trivial on Zassenhaus subgroups} and since $\sigma \in N \leq G_{(n)}$, the image of $\rho(\bar{\sigma})$ in $\overline{\U}(\sA)$ is trivial. Since also  $\rho(\bar{\sigma})_{1,n+1}=0$, this implies that 
 $\rho(\bar{\sigma}) = 1$. Hence $\rho_0(\sigma) = 1$, i.e., $\sigma \in \ker \rho_0$. Because (2) holds, we conclude that $\sigma \in G_{(n+1)}$ and thus $\bar{\sigma}$ is trivial in $N / N \cap G_{(n+1)}$. 
\end{proof}

\begin{thm} 
\label{thm:kernel-pairing}
 Let $G$ be a pro-finite group. The following two statements are equivalent.
\begin{enumerate}
\item For every $k=2,\ldots,n$, the pairing
\[G_{(k)}/G_{(k+1)} \times \ker (\Phi^k(G/G_{(k)})\to \Phi^k(G/G_{(k+1)}))\to \F_p\] 
is non-degenerate.
\item For every $k=1,\ldots, n$, one has $G_{(k+1)}=\cap \ker (\rho\colon G\to \U(\sA))$, where the intersection is taken as $\sA$ runs over the collection of all rank $k$ multiplicative systems with $A_{1,k+1}=\F_p$ and $\rho$ runs over all homomorphism $G\to \U(\sA)$.
\end{enumerate}
\end{thm}
\begin{proof}
By Proposition~\ref{prop: non-degenerate and interesection} we see that (2) implies (1).

Now suppose that $(1)$ holds true. For a fixed $n\geq 2$, we proceed on induction on $k$, $1\leq k\leq n$, to show that $G_{(k+1)}=\cap \ker (\rho\colon G\to \U(\sA))$.
The statement is true for $k=1$, because we 
always have $G_{(2)} = \bigcap \ker(\rho\colon G \to \U(\sA))$ by Remark~\ref{rmk:2.5}. We suppose that $2\leq k\leq n$.
 By the induction hypothesis and by Lemma~\ref{lem: trivial on Zassenhaus subgroups}, it is enough to show that 
\[ G_{(k+1)} \supseteq \bigcap \ker (\rho:G\to \U(\sA)),\]
 where the intersection is taken as $\sA$ runs over the collection of all rank $k$ multiplicative systems with $A_{1,k+1}=\F_p$ and $\rho$ runs over all homomorphism $G\to \U(\sA)$. Equivalently, we shall show that for any $\sigma\not\in G_{(k+1)}$,  there exists a rank $k$ multiplicative system  $\sA$ with $A_{1,k+1}=\F_p$  and a homomorphism $\rho\colon G\to \U(\sA)$ such that
$\rho(\sigma)\not=1$. 

Let $\sigma\not\in G_{(k+1)}$. If $\sigma\not \in G_{(k)}$ then by the induction hypothesis,  there exists a multiplicative system of rank $k-1$ and a homomorphism $\rho\colon G\to \U(\sA)$ such that $\rho(\sigma)\not=1$. By Remark~\ref{rmk:2.2},  there exists a multiplicative system $\mathcal{A}$ of rank $k$ and an injective homomorphism $\iota\colon \U(\overline{\mathcal{A}}) \to \U(\mathcal{A})$. Let $\rho \colon G \to \U(\mathcal{A})$ be the composition of $\bar{\rho}$ and $\iota$. Then we have $\rho(\sigma) \neq 1$.

Now we assume that $\sigma \in G_{(k)}\setminus G_{(k+1)}$. 
 Since $\sigma$ is nontrivial in $G_{(k)}/G_{(k+1)}$ and by the non-degenerateness of the pairing in (1), there is $\alpha\in \ker(\Phi^2(G/G_{(k)})\to \Phi^2(G/G_{(k+1)}))$ such that $\langle \sigma, \alpha\rangle\not=0$.
From the definition of $\Phi^k(G/G_{(k)})$, there exists a rank $k$ multiplicative system $\sA$ with $A_{1,k+1}=\F_p$ and a group homomorphism $\bar\rho\colon G\to \overline{\U}(\sA)$ which corresponds to $\alpha$. Since ${\rm inf}(\alpha)=0$ in $H^2(G/G_{(k+1)})$, there exists a lift $\rho\colon G/G_{(k+1)}\to \U(\sA)$ of $\bar{\rho}$. By Lemma~\ref{lem: pairing} (applying to $N=G_{(k)}$), one has 
\[
\rho(\sigma)_{1,k+1}=- \langle \sigma,\alpha\rangle\not =0.
\]
Hence $\rho(\sigma)\not =1$, and we are done.
\end{proof}

From now on, let $S$ be a free pro-$p$-group, $R$ be a normal subgroup of $S$ with $R\leq S_{(n)}$ and $G:=S/R$.
\begin{cor} 
\label{cor:perfect}
The pairing 
\[R/R\cap S_{(n+1)}\times \Phi^n(S/R)\to \F_p
\]
is non-degenerate.

In particular, the pairing
\[S_{(n)}/S_{(n+1)}\times \Phi^n(S/S_{(n)})\to \F_p
\]
is non-degenerate.
\end{cor}
\begin{proof} Since $S$ is free pro-$p$ group, we have $H^2(S)=0$. (See \cite[Theorem 4.12]{Ko}.) Hence, by Lemma~\ref{lem:3.1},
\[
\ker(\Phi^n(S/R)\to \Phi^n(S/R\cap S_{(n+1)}))=\ker(\Phi^n(S/R)\to H^2(S))= \Phi^n(S/R).
\] 
On the other hand, free groups have the intersection property by Theorem~\ref{thm:Efrat}, so the condition (2) in Proposition~\ref{prop: non-degenerate and interesection} holds for $G=S$. Hence  the condition (1) in Proposition~\ref{prop: non-degenerate and interesection} holds for $G=S$ and $N=R$. Thus the  pairing 
\[
R/R\cap S_{(n+1)}\times \Phi^n(S/R)\to \F_p
\] 
is non-degenerate. 
\end{proof}
\begin{rmk}
\label{rmk:Pontryagin}
Let the notation be as in the previous corollary. We equip the group $\Phi^n(S/R)$,  $\Phi^n(S/S_{(n)})$  (and $\F_p$) with the discrete topology. 
Recall that  $R/R\cap S_{(n+1)}$ is an abelian $p$-elementary compact group. We show below that $R/R\cap S_{(n+1)}$ and $\Phi^n(S/R)$ are in fact Pontryagin duals to each other. (See 
\cite[page 47]{Morris} for the definition of Pontryagin dual.)
The pairing is continuous and the non-degenerateness of the pairing of the compact ($p$-elementary) abelian group $R/R\cap S_{(n+1)}$ and the discrete ($p$-elementary)  abelian group $\Phi^n(S/R)$  implies that the pairing is actually perfect, i.e., the compact  group $R/R\cap S_{(n+1)}$ and the discrete  group $\Phi^n(S/R)$ are Pontryagin dual to each other.
 In fact, the induced map $\Phi^n(S/R)\to {\rm Hom} (R/R\cap S_{(n+1)}, \F_p)$ is open and injective. (Recall that a continuous map is open if  it maps open sets to open sets. In our case the map is open since both $\Phi^n(S/R)$ and ${\rm Hom} (R/R\cap S_{(n+1)}, \F_p)$ are discrete.) Hence the dual map $R/R\cap S_{(n+1)} \to {\rm Hom}(\Phi^n(S/R),\F_p)$ is surjective (by \cite[Proposition 30]{Morris}) and thus it is an isomorphism. Therefore the pairings in Corollary \ref{cor:perfect} are perfect.  See also \cite[Theorem 44]{Pon}, which says that any non-degenerate pairing between a compact group and a disrete group is perfect.
\end{rmk} 
\begin{cor} The pairing
 \[S_{(n)}/RS_{(n+1)} \times \ker(\Phi^n(S/S_{(n)}) \to \Phi^n(S/R))\to \F_p\] 
 is non-degenerate.
\end{cor}
\begin{proof}
Note that we have a natural isomorphism  $R/R\cap S_{(n+1)}\simeq RS_{(n+1)}/S_{(n+1)}$. By Corollary~\ref{cor:perfect} and Remark~\ref{rmk:Pontryagin}, the two rows of the following diagam provide perfect pairings
\[
\xymatrix{
\langle\cdot,\cdot\rangle_1\colon RS_{(n+1)}/S_{(n+1)} \ar@{} \ar@{->}[d]^{\alpha} &\times &   \Phi^n(S/R) \ar@{->}[r] &\F_p\\ 
\langle\cdot,\cdot\rangle_2\colon S_{(n)}/S_{(n+1)} \ar@{} &\times &   \Phi^n(S/S_{(n)}) \ar@{->}[r] \ar@{->}[u]^{\beta}&\F_p.\\
}
\]
This diagram is commutative in the following sense: for each 
\([r] \in RS_{(n+1)} / S_{(n+1)}\), \(r \in RS_{(n+1)}\), and each 
\(t \in \Phi^n(S / S_{(n)})\), we have
\[
\langle \alpha([r]), t \rangle_2 = \langle [r], \beta(t) \rangle_1.
\]

Here \(\alpha\) is the natural map induced by the inclusion 
\(RS_{(n+1)} \to S_{(n)}\), and \(\beta\) is the natural map induced by 
inflation as described in the beginning of Section 3 in the paragraph 
below the 5-term exact sequence.

By \cite[Lemma 2.3]{Ef}, the diagram induces a non-degenerate bilinear map 
\[ {\rm coker}(\alpha)\times {\rm ker}(\beta)\to \F_p.\]
 The statement follows. 

Alternatively, we can proceed as follows. We shall identify $\Phi^n(S / S_{(n)})$ (respectively, $\Phi^n(S/R)$) with the Pontryagin dual of  $S_{(n)}/S_{(n+1)}$ (respectively, $RS_{(n+1)}/S_{(n+1)}$). Let ${\rm Ann}({\rm im}\alpha)$ be the annihilator of ${\rm im}\alpha$: 
\[
{\rm Ann}({\rm im}\alpha)=\{y\in \Phi^n(S / S_{(n)}) \mid \langle z,y\rangle_2=0, \forall z\in {\rm im}\alpha\}.
\]
If $y\in \ker \beta$ and $z=\alpha x\in {\rm im} \alpha$, then $\langle z,y\rangle_2=\langle \alpha x,y\rangle_2=\langle x,\beta y\rangle_1=\langle x,0\rangle_1=0$. Hence $\ker \beta\subseteq {\rm Ann}({\rm im}\alpha)$. We show that in fact $\ker \beta = {\rm Ann}({\rm im}\alpha)$. Suppose that $y\not \in \ker \beta$, then $\beta y\not=0$. Since $\langle \cdot, \cdot \rangle_1$ is non-degenerate, there exists $x\in RS_{(n+1)}/S_{(n+1)} $ such that $\langle x,  \beta y \rangle_1\not=0$. Hence
\[
0\not= \langle x,  \beta y \rangle_1= \langle \alpha x,  y \rangle_2.
\]
This implies that $y\not \in {\rm Ann}({\rm im}\alpha)$. Thus $\ker \beta = {\rm Ann}({\rm im}\alpha)$. 
By \cite[Theorem 2.17]{Morris}, we see that $\ker\beta$ and ${\rm coker}(\alpha)$ are Pontryagin duals to each other. The pairing  ${\rm coker}(\alpha)\times {\rm ker}(\beta)\to \F_p$ is indeed perfect.
\end{proof}

\begin{cor}
\label{cor:non-degenerate pairing}  Let $S$ be a free pro-$p$-group and let $R$ be a normal subgroup of $S$ with $R\leq S_{(n)}$. Let $G=S/R$. 
Then the pairing 
\[
G_{(n)}/G_{(n+1)} \times  \ker(\Phi^n(G/G_{(n)})\to \Phi^n(G))\to \F_p
\]
is non-degenerate.
\end{cor}
\begin{proof} The following diagram of exact sequences
\[
\xymatrix{
1\ar@{->}[r]  & S_{(n)}/RS_{(n+1)} \ar@{->}[r] \ar@{->}[d]_{\simeq} &S/RS_{(n+1)}\ar@{->}[r] \ar@{->}[d]_{\simeq} & S/S_{(n)} \ar@{->}[r] \ar@{->}[d]_{\simeq} &1\\
1\ar@{->}[r] & G_{(n)}/G_{(n+1)} \ar@{->}[r] &G/G_{(n+1)}\ar@{->}[r] & G/G_{(n)} \ar@{->}[r]&1
}
\]
commutes. The result then follows from the previous corollary. 
\end{proof}

\begin{thm} 
\label{thm:main}
Let $n$ be an integer $\geq 1$. Let $S$ be a free pro-$p$-group and let $R$ be a normal subgroup of $S$ with $R\leq S_{(n)}$. Let $G=S/R$. Then \[ G_{(n+1)}=\bigcap \ker (\rho\colon G\to \U(\sA)),\] 
where the intersection is taken as $\sA$ runs over the collection of all rank $n$ multiplicative systems with $A_{1,n+1}=\F_p$ and $\rho$ runs over all homomorphism $G\to \U(\sA)$.
\end{thm}
\begin{proof} The case $n=1$ follows from Remark~\ref{rmk:2.5}. The case $n\geq 2$ follows from Theorem~\ref{thm:kernel-pairing} and Corollary~\ref{cor:non-degenerate pairing} .
\end{proof}
\section*{Acknowledgements}
We are grateful to the referee for his/her comments and valuable suggestions which
we have used to improve our exposition.


\end{document}